\documentclass[oneside,notitlepage,12pt]{article}

\usepackage{amssymb}
\usepackage[leqno]{amsmath}
\usepackage{amsfonts}
\usepackage{amsopn}
\usepackage{amstext}
\usepackage{amsthm}

\usepackage[all]{xy}
\newdir{ >}{{}*!/-9pt/@{>}}

\usepackage{calrsfs}
\usepackage{fourier-orns}
\usepackage{wasysym}
\usepackage{amssymb}
\usepackage[leqno]{amsmath}
\usepackage{amsfonts}
\usepackage{amsopn}
\usepackage{amstext}
\usepackage{amsthm}
\usepackage{enumitem}
\usepackage[dvipsnames]{xcolor}

\usepackage[colorlinks]{hyperref}
\usepackage{calrsfs}
\usepackage{fourier-orns}

\usepackage{clock} %\ClockFrametrue\ClockStyle2
\providecommand{\cal}{\mathcal}
\renewcommand{\Bbb}{\mathbb}

\newenvironment{pf}{\begin{proof}}{\end{proof}}

\newcommand{\Ef}{{\cal{F}}}

\newcommand{\Pee}{{\cal{P}}}

\newcommand{\Yu}{{\cal{U}}}

\newcommand{\Qyu}{{\Bbb{Q}}}
\newcommand{\Err}{{\Bbb{R}}}

\newcommand{\lam}{{\lambda}}

\newcommand{\eps}{\varepsilon}
\renewcommand{\phi}{\varphi}
\renewcommand{\rho}{\varrho}

\newcommand{\rest}{\restriction}

\newcommand{\ntr}{{n\in\omega}}

\newcommand{\loe}{\leq}
\newcommand{\goe}{\geq}

\newcommand{\subs}{\subseteq}
\newcommand{\sups}{\supseteq}

\newcommand{\ext}{\operatorname{ext}}

\newcommand{\diam}{\operatorname{diam}}

\newcommand{\id}[1]{{\operatorname{id}_{#1}}} % identity morphism

\newcommand{\liminv}{\varprojlim}
\newcommand{\limind}{\varinjlim}
\newcommand{\by}{/}

\newcommand{\setof}[2]{\{#1\colon #2\}}

\newcommand{\sett}[2]{\{#1\}_{#2}}
\newcommand{\sn}[1]{\{#1\}} % singleton
\newcommand{\dn}[2]{\{#1,#2\}} % doubleton
\newcommand{\pair}[2]{\langle #1, #2 \rangle} % pair
\newcommand{\triple}[3]{\langle #1, #2, #3 \rangle} % triple
\newcommand{\map}[3]{#1\colon #2 \to #3} % A function
\newcommand{\img}[2]{#1[#2]} % image of a set
\newcommand{\inv}[2]{{#1}^{-1}[#2]} % preimage of a set

\newcommand{\fra}{Fra\"iss\'e}

\providecommand{\nat}{\omega}

\newcommand{\ciag}[1]{{\sett{{#1}_n}{\ntr}}}

\newcommand{\fS}{{\mathfrak{S}}}

\newcommand{\Ob}{{\mathrm{Ob}}}

\newcommand{\Sim}{\rp \fS}
\newcommand{\cmp}{\circ} % composition!!!

\newcommand{\invsys}[5]{\langle {#1}_{#4};{#2}_{#4}^{#5};#3 \rangle}

\newcommand{\rp}[1]{{\ddag{#1}}} % the category of retractive pairs

\newcommand{\ob}[1]{\operatorname{Ob}\left(#1\right)}

\newcommand{\proto}[1]{{\mathbb S_\kappa}}

\newtheorem{tw}{Theorem}[section]
\newtheorem{wn}[tw]{Corollary}
\newtheorem{lm}[tw]{Lemma}
\newtheorem{prop}[tw]{Proposition}

\theoremstyle{definition}

\theoremstyle{remark}
\newtheorem{uwgi}[tw]{Remark}
\newcommand{\Komp}{\mathfrak K\mathfrak o\mathfrak m\mathfrak p}
\newcommand{\cplus}{\rp{\mathfrak C}}
\newcommand{\Fans}{\rp {\mathfrak F}}
\newcommand{\Lelfan}{{\mathbb L}}

\newcommand{\Cantor}{2^\omega}
\newcommand{\koment}[1]{}

\title{The Lelek fan and the Poulsen simplex as \fra\ limits}
\author{
{\sc Wies{\l}aw Kubi\'s}\footnote{Research of the first author supported by GA\v CR grant P201/12/0290 and RVO 67985840.}\\ \\
{\small Institute of Mathematics, Czech Academy of Sciences}\\
{\small Prague, Czech Republic}\\
{\small {\it and}}\\
{\small Faculty of Mathematics and Natural Sciences, College of Science}\\
{\small Cardinal Stefan Wyszy\'nski University}\\
{\small Warsaw, Poland}
\and
{\sc Aleksandra Kwiatkowska}\\ \\
{\small Institute of Mathematics, University of  Bonn} \\
{\small Bonn, Germany}
{\small {\it and}}\\
{\small Instytut Matematyczny, Uniwersytet Wroc{\l}awski}\\
{\small Wroc{\l}aw, Poland}
}
\date{\today\ \clocktime}

\begin{document}

\maketitle

\begin{abstract}
We describe the \emph{Lelek fan}, a smooth fan whose set of end-points is dense, and the \emph{Poulsen simplex}, a Choquet simplex whose set of extreme points is dense,
 as \fra\ limits in  certain natural categories of embeddings and projections.
 As an application we give a short proof of their uniqueness, universality, and almost homogeneity.
We further show that for every two countable dense subsets of end-points of the Lelek fan there exists an auto-homeomorphism of the fan mapping one set onto the other.
This improves a result of Kawamura, Oversteegen, and Tymchatyn from 1996.

\ \\
\noindent \textit{Keywords:}
Lelek fan, smooth fan, retraction, \fra\ limit.

\noindent \textit{MSC (2010):}
18B30,  %(1980-now) Categories of topological spaces and continuous mappings
54B30,  %(1973-now) Categorical methods
54F50,  %(1973-now) Spaces of dimension ≤1; curves, dendrites
54C15.  %(1973-now) Retraction
\end{abstract}

\tableofcontents

\section{Introduction}

The theory of \emph{\fra\ limits} has been recently extended beyond first-order structures, in particular, covering some topological spaces.
For example, one can consider a class of ``small"  (for example, compact) topological spaces with certain properties, looking at inverse limits coming from sequences in the fixed category of topological spaces.
Under some natural conditions, there exists a universal space that maps onto all spaces from the class and satisfies a variant of ``surjective homogeneity", implying that its group of auto-homeomorphisms is rich.
Such a space is unique and is called the \emph{\fra\ limit} of the class under consideration.
Of course, the details are more complicated, we refer to \cite{K_flims}, \cite{KubMF}, and \cite{IrwSol} for more information.

In this note we first consider a very special and simple class of compact metric spaces, namely \emph{finite fans}, together with a very particular class of continuous mappings.
We recognize its \fra\ limit as the \emph{Lelek fan}, a well-known object in geometric topology, constructed by Lelek~\cite{Lelek} in 1961, whose uniqueness was not discovered for almost 30 years.

The Lelek fan has been recently studied by Barto\v sov\'a and Kwiatkowska  \cite{BarKwi} in the framework of the projective \fra\ theory developed by Irwin and Solecki  \cite{IrwSol}.
Our aim is to describe the Lelek fan as a \fra\ limit in the sense of Kubi\'s \cite{KubMF}.
We are going to do it by developing a natural geometric structure of smooth fans.
This will  allow us to strengthen a few results from \cite{BarKwi} and to prove new ones.

The Poulsen simplex, constructed  by Poulsen \cite{Poul} in the same year as the Lelek fan, was studied in the 70s by  Lindenstrauss,  Olsen, and Sternfeld \cite{LOS}, who showed its uniqueness, universality, and homogeneity.
The space of real-valued continuous affine functions on the Poulsen simplex, from which one can recover the Poulsen simplex as the space of states,
 was constructed  in a  Fra\"{i}ss\'e-theoretic framework by Conley and T\"ornquist (unpublished). 
Independently from our work, constructions of the Poulsen simplex as
a \fra\ limit  were recently given in the work of Barto\v sov\'a, Lopez-Abad, and Mbombo, and in the paper by Lupini \cite{L}, who uses the framework of  the
\fra\ theory for metric structures
in the sense of Ben Yaacov \cite{IBY}, and for that he works with  the category of real-valued continuous affine functions  on Choquet simplices
which forms a  category dual to the one of  Choquet simplices.  %page 20 FTT
In fact, Lupini shows how to construct a number of objects from functional analysis as \fra\ limits. 

We describe the Poulsen simplex in the Fra\"{i}ss\'e-theoretic framework developed by Kubi\'s \cite{KubMF}. The advantage of such approach is that we work directly with simplices and affine maps between them,
rather than in the dual category. 

We would like to point out that the  \fra\ families we will consider to obtain the Lelek fan and to obtain the Poulsen simplex are very similar,
and in both cases the morphisms will be affine projections.

The set $E$ of end-points of the Lelek fan is an interesting topological space, studied in detail by Kawamura, Oversteegen, and Tymchatyn.
One of their results~\cite[Thm. 12]{Kawamuraetal} states that for every two countable dense subsets of $E$ 
 there is an auto-homeomorphism of $E$ moving the first set onto the other.
We generalize their theorem  and
show that for every two countable dense subsets of $E$ 
there exists an auto-homeomorphism of the full Lelek fan which is  linear on each spike and   moves the first set onto the other.
We further present an example of an auto-homeomorphism of $E$ that does not extend to the Lelek fan,
which shows that our result does not follow from their result.

\section{Preliminaries}

Denote by $\Komp$ the category of all non-empty compact second countable spaces with continuous mappings.
We shall consider the category $\rp \Komp$ whose objects are non-empty compact metric spaces and arrows are pairs of the form $\pair e p$, where $e$ is a continuous injection and $p$ is a continuous surjection %(often later called an affine projection) 
satisfying $p \cmp e = \id {A}$, where $A$ is the domain of $e$.
The composition  is $\pair e p \cmp \pair {e'}{p'}$ is $\pair {e \cmp e'}{p' \cmp p}$.
The \emph{domain} of an arrow $\pair e p$ is, by   definition, the domain of $e$.
Its \emph{co-domain} is the domain of $p$.

We will review the Fra\"{i}ss\'e-theoretic framework introduced by  Kubi\'s \cite{KubMF}, focusing only on subcategories of  $\Komp$.
Let $\mathfrak{C}$ be a subcategory of $\Komp$ and let $\cplus$ be a subcategory of $\rp \Komp$ such that 
$\Ob(\mathfrak{C})=\Ob(\mathfrak{\cplus})$.
Let us assume that each $F\in \Ob(\mathfrak{C})$ is equipped with a metric $d_F$.
Given two $\cplus$-arrows $\map{f,g} F G$, $f = \pair e p$, $g = \pair i q$, we define
$$d(f,g) =
\begin{cases}
\max_{y \in G}d_F(p(y),q(y)) & \qquad \text{if }e = i,\\
+\infty & \qquad \text{otherwise}.
\end{cases}$$
This defines a metric on the set of all arrows from $F$ to $G$.
Using the same formula, we define a metric on the set of all arrows from $F$ to $G$, where $F\in \Ob(\mathfrak{C})$ 
and $G$ is (the limit of) a $\cplus$-sequence; it will appear in results about almost homogeneity.
Say that $\cplus$ equipped with the metric $d$ is a \emph{metric category} if
$d(f_0\circ g,f_1\circ g)\leq d(f_0,f_1)$ and $d(h\circ f_0,h\circ f_1)\leq d(f_0,f_1)$, whenever the composition makes sense.

We say that $\cplus$ is \emph{directed} if for every $A,B\in\cplus$ there is $C\in\cplus$ such that there exist arrows from $A$ to $C$ and from $B$ to $C$.
Say that $\cplus$ has the \emph{almost amalgamation property} if for every $\cplus$-arrows $f\colon C\to A$, $g\colon C\to B$, for every $\eps>0$, there exist
$\cplus$-arrows $f'\colon A\to W$, $g'\colon B\to W$ such that $d(f'\circ f,g'\circ g)<\eps$. It has the  \emph{strict amalgamation property} if we can have $f'$ and $g'$ as above
satisfying $f'\circ f=g'\circ g$.

The category $\cplus$ is \emph{separable} if there is a countable subcategory $\Ef$ such that
\begin{enumerate}[itemsep=0pt]
	\item[(1)] for every $X \in \ob{\cplus}$ there are $A \in \ob{\Ef}$ and an arrow $\map{f}{X}{A}$;
	\item[(2)] for every $\cplus$-arrow $f\colon A\to Y$ with $A\in \Ob(\Ef)$. for every $\eps>0$ there exists an $\cplus$-arrow  $g\colon Y\to B$ and an $\Ef$-arrow $u\colon A\to B$ such that
	$d(g\circ f, u)<\eps$.
\end{enumerate}

We say that a $\cplus$-sequence $\vec U =  \invsys U u \omega m n$ is \emph{a \fra\ sequence} if the following holds:
\begin{enumerate}
\item[(F)] Given $\eps>0$, $m\in\omega$, and an arrow  $\map f {U_m}F$, where $F\in \Ob(\cplus)$, there exist $m<n$ and an arrow  $\map g F {U_n}$
such that $d(g \cmp f, u_m^n) < \eps$.
\end{enumerate}

\begin{tw}[Thm. 3.3, \cite{KubMF}]\label{fraseq}
 Let $\cplus$ be a directed  metric subcategory of $\rp\Komp$ with the almost amalgamation property. The following conditions are equivalent:
 \begin{itemize}
\item[(a)] $\cplus$ is separable.
\item[(b)] $\cplus$ has a \fra\ sequence.
\end{itemize}
\end{tw}

We have the following general theorems  for a metric category $\cplus$ as above. All the necessary definitions are given on pages 5-7 in \cite{KubMF} .

\begin{tw}[Uniqueness, Thm.  3.5, \cite{KubMF}]\label{uniq}
There exists at most one  \fra\ sequence (up to an isomorphism).
\end{tw}

\begin{tw}[Universality, Thm. 3.7, \cite{KubMF}]\label{univ}
Suppose that $\vec{U}$ is a \fra\ sequence  in $\cplus$.
Then for every sequence  $\vec{X}$ in $\cplus$ there is an arrow $f\colon \vec{X}\to \vec{U}$.
\end{tw}
%(see definitions on pages 5 and 6)

\begin{tw}[Almost Homogeneity, Thm. 3.6, \cite{KubMF}]\label{alhom}
Suppose that $\cplus$ has the almost amalgamation property and it has a \fra\  sequence $\vec{U}$.
Then for every $A,B\in \ob{\cplus}$ and  for all arrows $i\colon A\to \vec{U}$, $j\colon B\to \vec{U}$, 
for every $\cplus$-arrow $f\colon A\to B$, for every $\eps>0$, there exists an isomorphism $H\colon \vec{U}\to\vec{U}$ such that 
$d(j\circ f, H\circ i)<\eps$.
\end{tw}

In particular, we can take above $A=B$ and $f=\id{}$.

In Section \ref{sub-Lprop}, we will work with a slightly more general category than $\rp \Komp$.
Namely, we will consider a subcategory of $\rp\Komp^d$, the category whose objects are pairs of sets $(F^1,F^2)$ with $F^1\subset F^2$, and $F^2$ compact.
An arrow from $(F^1,F^2)$ to $(G^1,G^2)$ in $\rp\Komp^d$ is a pair $\pair e p$ such that $e\colon F^1\to G^1$ is an injection, 
$ p\colon G^2\to  F^2$ is a continuous surjection so that $p \cmp e = \id{F^1}$ (more formally: $(p \rest G^1) \cmp e$ is the inclusion $F^1 \subs F^2$).
All definitions and theorems given above will remain true in this slightly more general setting, as this still will fall under the setting developed in \cite{KubMF}.

A \emph{retraction} is a continuous mapping $\map r X X$ satisfying $r \cmp r = r$.
A mapping $\map f X Y$ is \emph{right-invertible} if there exists a mapping $\map i Y X$ such that $f \cmp i = \id Y$.
Note that in this case $r = i \cmp f$ is a retraction.
On the other hand, if $\map r X X$ is a retraction then $r = i \cmp f$, where $i$ is the inclusion of $Y = \img f X$ and $f$ is the mapping $r$ treated as a surjection onto $Y$.
We shall often speak about retractions, having in mind right-invertible mappings.

\begin{lm}\label{Lembgerg}
Assume that $K$ is a compact space and $\ciag r$ is a sequence of retractions of $K$ that is pointwise convergent to the identity and satisfies $r_n \cmp r_m = r_{\min(n,m)}$ for every $n,m \in \nat$.
Then $K$ is the inverse limit of the sequence consisting of $K_n := \img{r_n}{K}$, such that the bonding map from $K_{n+1}$ to $K_n$ is $r_n \rest {K_{n+1}}$.
\end{lm}

\begin{pf}
Let $P = \prod_{\ntr}K_n$ and define $\map h K P$ by setting $h(x) = (r_n(x))_{\ntr}$.
As $\ciag r$ is pointwise convergent to the identity, $h$ is one-to-one.
The condition $r_n \cmp r_m = r_{\min(n,m)}$ ensures that $\img h K$ is inside the inverse limit $L \subs P$ of the sequence $\vec K = \invsys K r \omega n m$.
Fix $y \in L$. Then $y = (x_n)_{\ntr}$, where $x_n = r_n^m(x_m) = r_n(x_m)$, whenever $n<m$.
Let $x$ be an accumulation point of $\ciag x \subs K$.
Then $r_n(x) = x_n$ for every $\ntr$, therefore $h(x)= y$.
This shows that $\map h K L$ is a homeomorphism.
\end{pf}

\begin{uwgi}\label{RemDwaSzec}
The converse to the above lemma holds as well.
Namely, every inverse sequence of compact spaces $K_n$ whose bonding maps are right-invertible can be turned into a chain $K_0 \subs K_1 \subs \dots$ so that the limiting projections become retractions satisfying the assumptions of Lemma~\ref{Lembgerg}.
The proof is an easy exercise.
\end{uwgi}

For all undefined notions from category theory we refer to~\cite{MacLane}.

\section{The Lelek fan}

In this section we present category-theoretic framework for the class of smooth fans, showing that the Lelek fan is a \fra\ limit.
Next we prove the announced result involving countable dense sets of end-points of the Lelek fan.

\subsection{Geometric theory of smooth fans}

Our universe is the topological space $V = \Delta(\Cantor) = ([0,1] \times \Cantor) \by \sim$, where $2^\omega$ denotes the Cantor set and $\sim$ is the equivalence relation identifying all points of the form $\pair 0 t$, $t \in 2^\omega$ and no other elements of the space.
In other words, $V$ is the cone over the Cantor set.
Obviously, $V$ is embeddable into the plane in such a way that the top of the cone is $0$ and for each $t \in 2^\omega$ the image of the set $[0,1] \times \sn t$ is a straight line segment.
Thus, $V$ is closed under multiplication by non-negative scalars $\loe1$.
We say that a subset $S$ of $V$ is \emph{convex} if $\lam \cdot x \in S$ whenever $\lam \in [0,1)$ and $x \in S$. Of course, this notion has not much to do with the standard convexity in the plane.

The space $V$ is called the \emph{Cantor fan}.
Closed convex subsets of $V$ will be called \emph{geometric fans}.
\emph{Smooth fans} are defined in purely topological terms\footnote{
A \emph{fan} is an arcwise connected hereditarily unicoherent continuum with at most one ramification point, called the \emph{top}.
A fan is \emph{smooth} if for each sequence $(p_n)$ converging to $p$ the sequence of arcs $tp_n$ converges to the arc $tp$, where $t$ is the top point. See~\cite{Char} for more details and historical references.
}, up to homeomorphisms.
It turns out, however, that every smooth fan is homeomorphic to a geometric fan~\cite{Eberhart} and clearly every geometric fan is smooth.
From now on we shall deal with geometric fans only, having in mind that we actually cover the class of all smooth fans.

We shall use the function $\map \rho V [0,1]$ defined by $\rho(s,t) = s$.
We  call $\rho(x)$ the \emph{level of} $x \in V$.
We say that a function $\map f S V$, where $S \subs V$, is \emph{level-preserving} if $\rho(f(s)) = \rho(s)$ for every $s \in S$.
Given a fan $F$, we denote by $E(F)$ the set of its end-points.
Formally, $x \in F$ is an \emph{end-point} if there are no $y\in F$ and $\lam \in [0,1)$ such that $x = \lam \cdot y$.
Obviously, $E(V) = \sn 1 \times 2^\omega$ and every fan $F$ with $E(F)$ finite is actually homeomorphic to the cone over its end-points.
We shall say that $F$ is a \emph{finite fan} if $E(F)$ is finite.
Given $x \in E(F)$, the minimal convex set containing $x$ is the line segment joining it to $0$.
Such a set will be called a \emph{spike}.

Given two geometric fans $F$, $G$, a map $\map f F G$ will be called \emph{affine} if it is continuous and $f(\lam \cdot x) = \lam \cdot f(x)$ for every $x \in F$, $\lam \in [0,1)$.
If $F$ is finite, then the last condition actually implies continuity.

Given a finite geometric fan $F$ embedded into a plane, we define a metric $d_F$ on $F$. Let $d_F(x,y)$ be the length of the shortest path
(which  is always either a line segment or the union of  two line segments) between $x$ and $y$, 
computed with respect to the usual Euclidean metric on the plane, in which $V$ is embedded.
This is obviously a metric compatible with the topology, as long as $F$ is a finite fan.
Later on, we shall consider 1-Lipschitz affine mappings between finite fans, always having in mind the above metrics.

Let $\Delta(X)$ denote the cone over the space $X$.
Fix a geometric fan $K$.
We shall analyze affine quotients of $K$ onto finite fans.
Recall that $K$ ``lives" in the Cantor fan $V = \Delta(\Cantor)$.
A subset $U \subs K$ will be called \emph{a triangle} if it is of the form $U = K \cap \Delta(W)$, where $W \subs \Cantor$ is a non-empty basic clopen set, that is, $W = W_s$, where
$$W_s = \setof{t \in \Cantor}{s \subs t}$$
and $s \in 2^{<\omega}$.
($2^{<\omega}$ denotes, as usual, the set of all finite sequences with values in the set $2 = \dn 01$).
In particular, two triangles are either disjoint or one is contained in the other.
The \emph{size} of a triangle $U$ as above is defined to be $1/n$, where $n$ is the cardinality (length) of the finite sequence $s$.
Note that every triangle $U$ is almost clopen, namely, $U$ is closed and $U\setminus \sn0$ is open.
Now, a \emph{triangular decomposition} $\Yu$ 
of $K$ is a (necessarily finite) family $\Yu$ consisting of triangles, such that $\setof{\img\pi U}{U\in \Yu}$ is a partition of $\img \pi K \subs \Cantor$.
Formally, a triangular decomposition is a covering of $K$ by almost clopen sets and it becomes a partition only after removing the vertex $0$.

Given a triangle $U$, define $\rho(U) = \max\setof{\rho(x)}{x\in U}$.
Now suppose $\Yu$ is a triangular decomposition of $K$ and choose for each $U \in \Yu$ some $a_U \in E(K)$ with $\rho(a_U) = \rho(U)$.
Let $K_0$ be the finite fan whose set of end-points is $\setof{a_U}{U \in \Yu}$.
There is a unique level-preserving continuous affine retraction $\map r K {K_0}$ satisfying $\inv f {[0,a_U]} = U$ for every $U \in \Yu$.

Using the observations above, we can show the following proposition.
\begin{prop}\label{LmfgrYgnX}
Every geometric fan is the inverse limit of a sequence of finite geometric fans whose bonding mappings are affine and level-preserving (therefore 1-Lipschitz).
\end{prop}

\begin{pf}
We  choose a sequence of triangular decompositions $\Yu_n$ such that the size of each triangle in $\Yu_n$ is not more than $1/n$, and $\Yu_{n+1}$ refines $\Yu_n$.
For each $\Yu_n$ we take a level-preserving  affine retraction as in the remarks above. This gives the required sequence.
\end{pf}

We are interested in a better representation of geometric fans.
An \emph{embedding} of geometric fans is, by definition, a one-to-one continuous affine map. By compactness, this is obviously a topological embedding.
An embedding $\map e F G$ will be called \emph{stable} if $\img e {E(F)} \subs E(G)$.
If $e$ is the inclusion then this just means that $G$ is obtained from $F$ by adding some new spikes and not enlarging the existing ones.

\begin{lm}\label{LmRaTxGh}
Let $G$ be a finite geometric fan and let $\map q F G$ be an affine retraction between geometric fans.
Then there exists a stable embedding $\map e G F$ such that $q \cmp e = \id G$.
\end{lm}

\begin{pf}
For each $x \in E(G)$ choose $y_x \in F$ such that $q(y_x) = x$.
Necessarily $y_x \in E(F)$, because $q$ is linear on each spike.
Define $\map e G F$ to be the unique affine map such that $e(x) = y_x$.
Continuity comes automatically from the fact that $G$ is a finite fan.
Thus, $e$ is a stable embedding.
That $q \cmp e = \id G$ follows from the fact that a linear self-map of a closed interval fixing its end-points must be the identity.
\end{pf}

Combining Proposition \ref{LmfgrYgnX},  Lemma \ref{LmRaTxGh} and Remark~\ref{RemDwaSzec}, we obtain the following theorem.

\begin{tw}\label{ThmMejnRrr}
Let $F$ be a geometric fan.
Then there exists a chain
$$F_0 \subs F_1 \subs F_2 \subs \dots \subs F$$
of finite fans, such that $E(F_n) \subs E(F)$ for each $\ntr$, and there exist 1-Lipschitz affine retractions $\map {r_n}F{F_n}$ such that $r_n \cmp r_m = r_{\min(n,m)}$ for every $n,m \in \nat$, and $\ciag r$ converges pointwise to $\id F$.
\end{tw}

\subsection{Construction and properties of the Lelek fan}

We now have all the necessary tools for describing the Lelek fan and its properties.
Recall that each finite geometric fan $F$ has a fixed metric $d_F$, defined as the shortest path between two points, assuming that each spike has the length inherited from its fixed embedding into the Euclidean space.

Let $\Fans$ be the category whose objects are finite geometric fans and an arrow from $F$ to $G$ is a pair $\pair e p$ such that $\map e F G$ is a stable embedding, $\map p G F$ is a 1-Lipschitz affine retraction and $p \cmp e = \id F$.
The composition $\pair e p \cmp \pair {e'}{p'}$ is $\pair {e \cmp e'}{p' \cmp p}$.
An immediate consequence of Theorem~\ref{ThmMejnRrr} is

\begin{wn}\label{geoinv}
	Every geometric fan is the inverse limit of a sequence in $\Fans$.
\end{wn}

\begin{lm}
The category $\Fans$ has the strict amalgamation property.
\end{lm}
\begin{proof}
Let $A,B,C$ be finite geometric fans such that $C = A \cap B$.
Suppose also that two 1-Lipschitz affine retractions $\map r A C$, $\map s B C$ are given.
Obviously, $A \cup B$ is a finite geometric fan and there are 1-Lipschitz affine retractions $\map {r'}{A\cup B}{A}$, $\map {s'}{A\cup B}B$ defined by conditions
$$r' \rest B = s, \quad r' \rest A = \id A, \qquad s' \rest A = r, \quad s' \rest B = \id B.$$
Now $f'=\pair {e^{A\cup B}_A}{r'}$ and $g'=\pair {e^{A\cup B}_B}{s'}$
 provide an amalgamation of the pairs $f = \pair {e^A_C}r$ and $g = \pair {e^B_C}s$ in the category $\Fans$, where $e_C^A$, $e_C^B$ denote the inclusions $C \subs A$, $C \subs B$.
\end{proof}

\begin{lm}
$\Fans$ is a separable metric category.
\end{lm}

\begin{pf}
Let $\Ef$ consist of all finite fans such that $\rho$ restricted to the end-points has rational values.
Let us consider all $\Fans$-arrows $\pair e p$ such that $p$ maps the end-points into points with rational value of $\rho$ (recall that $\rho(x)$ is the first coordinate of $x$ in $V$).
After identifying isometric fans, this family of objects and arrows is certainly countable.
\end{pf}

By Theorem \ref{fraseq}, $\Fans$ has a \fra\ sequence. The next result characterizes its limit.

\begin{tw}\label{dense}
Let $\vec U$ be a sequence in $\Fans$ and let $U_\infty$ be its inverse limit.
The following properties are equivalent:
\begin{enumerate}
	\item[(a)] The set $E(U_\infty)$ is dense in $U_\infty$.
	\item[(b)] $\vec U$ is a \fra\ sequence.
\end{enumerate}
\end{tw}

\begin{pf}

{\bf{(a) implies (b)}}
Fix $m\in\omega$ and $\eps>0$.
Fix a $\Fans$-arrow $\map f {U_m}{U}$, with $f = \pair e p$.
It suffices to prove condition (F) in case where $U$ is an extension of $U_m$ by adding a single spike.
All other $\Fans$-arrows are obtained as compositions of such arrows.
Let $v$ be the ``new'' end-point of $U$ and let $x = p(v)$.
There exists a sequence $\ciag x$ of end-points of $U_\infty$ convergent to $x$.
Then $\lim_{k\to\infty} r_m(x_k) = r_m(x) = x$.
As $U_m$ is a finite fan, for $k$ big enough we have that $r_m(x_k)$ is $\eps$-close to $x$ with respect to the metric $d$ defined for finite fans.
Finally, the fan $W$ obtained from $U_m$ by adding the spike with end-point $x_k$ ``realizes'' $U$ with an $\eps$-error.
Making another $\eps$-error, we can assume that $W \subs U_n$ for some $m < n$.

{\bf{(b) implies (a)}}
Fix $x \in U_m$, where $m$ is fixed.
Let $n_0 = m$.
Using condition (F), we find $n_1 > n_0$ and $x_1 \in E(U_{n_1})$ such that $d(r_{n_0}(x_1), x) < 2^{-1}$, where $d$ denotes the ``shortest path" metric.
Again using (F) applied to $U_{n_1}$, we find $n_2 > n_1$ and $x_2 \in E(U_{n_2})$ such that
$d(r_{n_1}(x_2), x) < 2^{-2}$.
We continue this way, ending up with a sequence of end-points $\sett{x_{k}}{k\in\nat}$ such that $\lim_{k\to \infty} r_{n_k}(x_{k+1}) = x$.
Refining this sequence, we may assume that it is convergent, i.e. $\lim_{k\to\infty} x_k = y \in U_\infty$ and $\sett{n_k}{k\in\nat}$ is a strictly increasing sequence of positive integers.
We need to show that $y = x$.

Fix $\ell \in \nat$.
If $n_k>\ell$ then $r_\ell = r_\ell \cmp r_{n_k}$, therefore
$$r_\ell(x) = r_\ell(\lim_{k\to \infty} r_{n_k}(x_{k+1})) = \lim_{k\to\infty}r_\ell(x_{k+1}) = r_\ell(y).$$
This shows that $r_\ell(x) = r_\ell(y)$ for every $\ell \in \nat$.
This is possible only if $x = y$.

Finally, as $\bigcup_{m \in \nat}U_m$ is dense in $U_\infty$, this shows the density of $E(U_\infty)$ in $U_\infty$.
\end{pf}

From now on, $\Lelfan$ will denote the limit of a \fra\ sequence in $\Fans$.

\begin{wn}
$\Lelfan$ is the Lelek fan.
\end{wn}

%, to emphasize that this is the Lelek fan and that it is unique.
\fra\ theory combined with our geometric theory of finite fans provides a simple proof of the uniqueness result, originally proved by Charatonik~\cite{Char} and independently by Bula \& Oversteegen~\cite{BuOv}.

\begin{wn}[Uniqueness]	
The Lelek fan is a unique, up to homeomorphisms, smooth fan whose set of end-points is dense.
\end{wn}

Further corollaries, coming from the  \fra\ theory are the following.

\begin{wn}[Universality]\label{CorSDFohe}
For every geometric fan $F$ there are a stable embedding into the Lelek fan $\Lelfan$ and a 1-Lipschitz affine retraction from $\Lelfan$ onto $F$.
\end{wn}

Corollary \ref{CorSDFohe} says, in particular, that the set of end-points of the Lelek fan is universal for the class of all sets of the form $E(K)$, where $K$ is a geometric fan.
This class was characterized in topological terms by Kawamura, Oversteegen and Tymchatyn \cite{Kawamuraetal}.

Lemma \ref{LmRaTxGh} and Theorem \ref{alhom} imply the following corollary.

\begin{wn}[Almost homogeneity]\label{almhom}
Let $F$ be a finite geometric fan, and let $p_1, p_2\colon \Lelfan\to F$ be continuous affine 1-Lipschitz surjections.
Then for every $\eps>0$  there is a homeomorphism $h\colon\Lelfan\to\Lelfan$ such that for every $x\in \Lelfan$,  $d_F(p_1\circ h(x), p_2(x))<\eps$.
\end{wn}

%Corollary~\ref{almhom} follows from the fact that every continuous affine surjection $p$ onto a finite fan has a right inverse, that is, $p \cmp e = \id{}$ for some affine embedding $e$.
 %Reasoning as in Theorem~4.25 in \cite{KubMF}, we obtain the following statement.
%Since the proof of Proposition \ref{LmfgrYgnX} gives that for every $\eps>0$ there exists 
%an affine projection $f\colon\Lelfan\to F$, 
%where $F$ is a finite fan,
%and all preimages of points in $F$ have the diameter not exceeding $\eps$ (where the metric on $\Lelfan$ we consider is inherited from the plane),
%reasoning as in Theorem~4.25 in \cite{KubMF},
%Corollary \ref{almhom} gives the following more general statement. 
In fact a stronger statement is true.
\begin{wn}\label{almhoms}
Let $F$ be a geometric fan (not necessarily finite) and let $p_1, p_2\colon \Lelfan\to F$ be continuous affine 1-Lipschitz 
surjections.
Then for every $\eps>0$  there is a homeomorphism  $h\colon\Lelfan\to\Lelfan$ such that for every $x\in \Lelfan$,  $d_F(p_1\circ h(x), p_2(x))<\eps$.
\end{wn}

\begin{proof}
Let ${\triple{{F}_n}{{q}_n^m}{\omega}}$ be an inverse sequence whose inverse limit is $F$, such that all bonding maps are affine and 1-Lipschitz,
constructed exactly as in the proof of Proposition \ref{LmfgrYgnX}. Each $F_n$ we view as a subset of $F$.
Take $n$ such that for every $x\in F$, we have $d_F(q^\infty_n(x), x)\leq \frac{\epsilon}{3}$.
From Corollary \ref{almhom} applied to $ q^\infty_n\circ p_i$ and $\frac{\epsilon}{3}$ we get 
$h\colon\Lelfan\to\Lelfan$ such that for every $x\in \Lelfan$,  $d_{F}(q^\infty_n\circ p_1\circ h(x), q^\infty_n\circ p_2(x))<\frac{\epsilon}{3}$.
This implies $d_F(p_1\circ h(x), p_2(x))<\frac{\epsilon}{3}+\frac{\epsilon}{3}+\frac{\epsilon}{3}=\eps$. 
\end{proof}

\subsection{More properties of the Lelek fan}\label{sub-Lprop}

We shall prove the following two statements:

\begin{tw}\label{Thmerbgwrw}
Let $\map f S T$ be a bijection, such that $S, T \subs E(\Lelfan)$ are finite sets.
Then there exists an affine homeomorphism $\map h \Lelfan \Lelfan$ such that $h \rest S = f$.
\end{tw}

\begin{tw}\label{ThmIeirgbiurhg}
Let $A, B \subs E(\Lelfan)$ be countable dense sets.
Then there exists an affine homeomorphism $\map h \Lelfan \Lelfan$ such that $\img h A = B$.
\end{tw}

The first result is rather well-known. 
A weakening of Theorem \ref{ThmIeirgbiurhg} was stated in  \cite[Thm. 12]{Kawamuraetal}, where the authors showed that the space of end-points of
the Lelek fan is  countably dense homogeneous.
We shall obtain both of these results from the general \fra\ theory, adding some extra work.

Given a sequence $\vec F$ in $\Fans$, we shall denote by $\liminv \vec F$ the inverse limit of the 1-Lipschitz affine retractions associated with $\vec F$.
We already know that $\liminv \vec F$ is a geometric fan and every geometric fan is of the form $\liminv \vec F$ for some sequence $\vec F$ in $\Fans$.
Given a sequence $\vec F$, let us denote by $\limind \vec F$ the union $\bigcup_{\ntr} F_n$, which is a dense subset of $\liminv \vec F$ consisting of countably many spikes.

\begin{pf}[Proof of Theorem~\ref{Thmerbgwrw}]
It will suffice to prove the theorem for $S$ and $T$ which are one-element sets, as then we  deduce the theorem in  a general case by decomposing the Lelek
fan into finitely many disjoint spaces homeomorphic to the Lelek fan, with their roots identified. 

Choose an affine (but not necessarily 1-Lipschitz) homeomorphism $h_S\colon \Lelfan\to \Lelfan'$ and  $h_T\colon \Lelfan\to \Lelfan''$ such that $h_S(S)$ is the 
endpoint of a longest spike in $\Lelfan'$
and  $h_T(T)$ is the endpoint of a longest spike in $\Lelfan''$. As in the proof of Proposition \ref{LmfgrYgnX}, we can find
sequences $\vec A$, $\vec B$ in $\Fans$ such that  $\Lelfan' = \liminv \vec A $ and $\Lelfan''= \liminv \vec B$, and additionally
$A_0 = h_S(S)$, $B_0 = h_T(T)$. By Theorem \ref{dense}, both sequences are Fra\"{i}ss\'e.
Corollary \ref{almhom}  applied to any $\eps >0$ gives an  affine isomorphism $f$ of $\Lelfan'$  and $\Lelfan''$ that carries $h_S(S)$ to $h_T(T)$. The homeomorphism $(h_T)^{-1}\circ f\circ h_S$ is as required.
\end{pf}

Let $F$ be a geometric fan.
A \emph{skeleton} in $F$ is a convex set $D \subs F$ such that $E(D)$ is countable, contained in $E(F)$ and dense in $E(F)$.

Let us consider the following generalization of the category $\Fans$.
Namely, let $\Fans^d$ be the category whose objects are pairs of  finite geometric fans $(F^1,F^2)$ with $F^1=F^2$,
and an arrow from $(F^1,F^2)$ to $(G^1,G^2)$ is a pair $\pair e p$ such that $e\colon F^1\to G^1$ is a stable embedding, 
$ p\colon G^2\to  F^2$ is a 1-Lipschitz affine retraction and $p \cmp e = \id F$.
Then $\Fans^d$ is a separable metric category, therefore it has a unique up to isomorphism \fra\ sequence.
Its \fra\ limit is $(D, \Lelfan)$ for some skeleton $D$ in $\Lelfan$.
To prove Theorem \ref{ThmIeirgbiurhg}, we will show the following crucial lemma, an analog for $\Fans^d$ of Corollary \ref{geoinv}.

\begin{lm}\label{LemCrucill}
Let $L$ be a geometric fan and let $D$ be a skeleton in $L$.
Then there exist a geometric fan $L'$, a skeleton $D'$ of $L'$, and 
an affine (not necessarily 1-Lipschitz) homeomorphism $h\colon L\to L'$ with $h(D)=D'$ such that 
there is a sequence $\vec F$ in $\Fans^d$ satisfying  $L' = \liminv \vec F$ and $D' = \limind \vec F$.
\end{lm}

We postpone for a moment the proof Lemma \ref{LemCrucill} and first show how to derive Theorem \ref{ThmIeirgbiurhg}.
The proof of the following theorem goes along the lines of the first part of the proof of Theorem \ref{dense}.

\begin{tw}\label{dense2}
Let $\vec U$ be a sequence in $\Fans^d$ such that $(D,\Lelfan)$ is its inverse limit, where $D$ is a skeleton of $\Lelfan$. Then
 $\vec U$ is a \fra\ sequence.
\end{tw}

\begin{pf}[Proof of Theorem~\ref{ThmIeirgbiurhg}]
Choose $A'$, $\Lelfan'$, and $h'\colon \Lelfan\to \Lelfan'$ with $h'(A)=A'$, and choose
 $B''$, $\Lelfan''$, and $h''\colon \Lelfan\to \Lelfan''$ with $h''(B)=B''$, as in  Lemma~\ref{LemCrucill}.
Then there are two sequences $\vec F$, $\vec G$ in $\Fans^d$ such that $\Lelfan' = \liminv \vec F$  and $\Lelfan''=\liminv \vec G$,
 $E(\limind \vec F) = A'$, and $E(\limind \vec G) = B''$.
By Theorem \ref{dense2}, both sequences are \fra\ in $\Fans^d$.
The uniqueness of a \fra\ sequence gives that there is an affine automorphism $\map H {\Lelfan'} {\Lelfan''}$ such that $\img H {A'} = B''$.
Then the map $  (h'')^{-1}\circ   H  \circ h' $ is as required.
\end{pf}

\begin{lm}\label{swap}
Let $U$ be a triangle in a geometric fan $L$. Let $0<b<1$. Let $e\in E(U)$ be such that $\frac{\rho(e)}{\rho(U)}>b$. 
Then there is an affine homeomorphism $h\colon U\to h[U]\subseteq U$ such that $\rho(h[U])=\rho(h(e))$
and for every $x\in U$, $1\geq \frac{\rho(h(x))}{\rho(x)}>b$. Moreover, $h(e)=e$ and for any $x\in U$, $h(x)$ is on the same spike as $x$.
\end{lm}
\begin{pf}
Let $U=V_0 \sups V_1 \sups V_2 \sups \cdots$ be a decreasing sequence of triangles such that $\bigcap_n V_n=[0,e]$.
Let $l_n=\rho(V_n)$. Let $h$ be such that $h\restriction (V_n\setminus V_{n+1})$ is the affine map
that takes $x\in V_n\setminus V_{n+1}$ into   $\frac{\rho(e)}{l_n}x$ and let $h\restriction[0,e]=\id{[0,e]}$.

Since $\lim_{n\to\infty} l_n=\rho(e)$, $h$ is continuous. Since moreover
$h$ is one-to-one, and $L$ is compact, we get that $h$ is a homeomorphism onto its image.
Note that $h$ satisfies all other required conditions.
\end{pf}

The remaining part of this section is devoted to the proof of Lemma~\ref{LemCrucill}.
We will need to modify the construction used  in the proof of Proposition~\ref{LmfgrYgnX}. 

\begin{pf}[Proof of Lemma~\ref{LemCrucill}]
Let $ \{s_n\}$ be an enumeration of $E(D)$.
We construct an affine homeomorphism $h\colon L\to L'\subseteq L$ 
and triangular decompositions $\widetilde{\mathcal{U}}_n$ of $L'$ of size $\frac{1}{n+2}$ 
such that for all $n$, for all $U\in \widetilde{\mathcal{U}}_n$, there is $a_U\in E(D)$ with $\rho(a_U)=\rho(U)$.
Once this is achieved, we will proceed as in the proof of Proposition \ref{LmfgrYgnX}.
Roughly speaking, we will replace
$L$ and $\{s_n\}$ by $L'$ and $\{h(s_n)\}$, for which  we can have level-preserving retractions.

First, we will do an inductive construction, where at step $n=-1,0,1,2,\ldots$ we will obtain
an affine homeomorphism $h_n\colon L_{n-1}\to L_n\subseteq L_{n-1}$ between  copies of $L$,
a triangular decomposition $\mathcal{U}_n$ of $L_n$ such that the size of each triangle in it is $\leq\frac{1}{n+2}$.
Furthermore, denoting $s'_i=h_n\circ\ldots\circ h_{-1} (s_i)$, $S'_n=\{s'_0, s'_1,\ldots, s'_n\}$, and $S'=\{s'_0,s'_1,\ldots\}$,
 for any $U\in \mathcal{U}_n$, there is $a_U\in E(U)\cap S'$ such that $\rho(a_U)=\rho(U)$, and each $s'\in S'_n$
 is realized as $a_U$ for some $U\in \mathcal{U}_n$. 

Fix a sequence of positive reals $\{b_n\}$ such that $\prod_n b_n>0$. 
\paragraph{Step $-1$.}
Let $\mathcal{U}_{-1}=\{L\}$ and let $h_{-1}\colon L_{-2}\to L_{-1}$, where  $L_{-2}= L_{-1}=L$, be
the identity map.

\paragraph{Step $n+1$.}
 Subdivide $\mathcal{U}_n$ to get $\mathcal{U}'_{n+1}$,
 a triangular decomposition of $L_n$, with
the property that the size of each triangle in it is $\leq\frac{1}{n+3}$, and moreover for $\widetilde{U}\in \mathcal{U}'_{n+1}$
such that $s'_{n+1}\in \widetilde{U}$, we have $\frac{\rho(s'_{n+1})}{\rho(\widetilde{U})}>b_{n+1}$. 
Let $a_{\widetilde{U}}=s'_{n+1}$ and
let for all $U\neq \widetilde{U}$, $a_U$ be some $s\in S'$ such that $\frac{\rho(s)}{\rho(U)}>b_{n+1}$.
We require that each $s'\in S'_{n}$ is realized as $a_U$ for some $U\in \mathcal{U}'_{n+1}$.
Apply Lemma \ref{swap} to each $U\in \mathcal{U}'_{n+1}$ (taking $e=a_U$ and $b=b_{n+1}$) and denote by $h_U$
the resulting affine homeomorphism. Let $h_{n+1}$ be such that $h_{n+1}\restriction U=h_U$, for every $U\in \mathcal{U}'_{n+1}$.
Let $L_{n+1}$ be the image of $h_{n+1}$ and  let $\mathcal{U}_{n+1}=\{U\cap L_{n+1}\colon U\in \mathcal{U}'_{n+1}\}$.
Let $a_U$ for $U\in \mathcal{U}_{n+1}$ be equal to $a_U$ for the corresponding $U\in \mathcal{U}'_{n+1}$.
Note that $h_{n+1}(a_U)=a_U$ for each $a_U$, $\rho(a_U)=\rho(U)$ for every $U\in \mathcal{U}_{n+1}$
and also each of $s''_0,\ldots, s''_{n+1}$ is realized as  $a_U$ for some $U\in \mathcal{U}_{n+1}$, where 
for each $i$,  $s''_i=h_{n+1}\circ\ldots\circ h_{-1} (s_i)$. 

\paragraph{The limit step.}
Let $h_n\colon L_n\to L_{n+1}$ be as above.
For each $n$ let $f_n=h_n\circ\ldots\circ h_0$. Clearly each $f_n$ is continuous and affine. 
The sequence $\{f_n\}$ converges uniformly, therefore its limit $h$ is a continuous function.
By the choice of $\{b_n\}$, it cannot happen that the image of a spike is equal to the root, and hence $h$ is one-to-one. Since $L$ is compact,
we get that $h$ is an affine homeomorphism onto its image, which we denote by $L'$. 
Let $\widetilde{\mathcal{U}}_n=\{U\cap L'\colon U\in \mathcal{U}_{n}\}$ and
let $a_U$ for $U\in\widetilde{\mathcal{U}}_n$ be equal to $a_U$ for the corresponding $U\in \mathcal{U}_{n}$.

We have  $\rho(a_U)=\rho(U)$ for every $U\in \widetilde{\mathcal{U}}_n$. Let $r_n$ be the retraction such that
$r_n\restriction U$, for every $U\in \widetilde{\mathcal{U}}_n$, is the level-preserving projection.
Note that as a set $\{h(s_n)\}$ is equal to the set set of all $a_U$, where $U\in \widetilde{\mathcal{U}}_n$ for some $n$.
Conditions of Lemma \ref{Lembgerg} are fulfilled.  
This completes the proof.
\end{pf}

\subsection{An example}

We finally present an example showing that Theorem~\ref{ThmIeirgbiurhg} cannot be deduced from the result of Kawamura, Oversteegen, and Tymchatyn~\cite{Kawamuraetal}
saying that for every countable dense sets $A,B\subset E(\Lelfan)$ there is a homeomorphism $g\colon E(\Lelfan)\to E(\Lelfan)$
such that $g(A)=B$. As we will see, not every homeomorphism of $E(\Lelfan)$ can be lifted to a homeomorphism of $\Lelfan$.

\begin{prop}\label{endp}
There exists a homeomorphism $h\colon E(\Lelfan)\to E(\Lelfan)$ such that for no homeomorphism $f\colon \Lelfan\to \Lelfan$, we have 
$f\rest E(\Lelfan)=h$.
\end{prop}

We assume below that the Cantor fan $V$ is the cone of a Cantor set $C$ contained in $[0,1]\times \{1\}$ over the top point $(\frac{1}{2},0)$.
 The Lelek fan $\Lelfan$ is realized as a subfan of $V$.
Without loss of generality, we can, and we will, assume that there is a spike in $\Lelfan$ joining the top point 
$(\frac{1}{2},0)$ with $(\frac{1}{2},1)$.
To prove Proposition \ref{endp}, since $E(\Lelfan)$ is dense in $\Lelfan$, it is enough to show Lemma \ref{endplem}.

\begin{lm}\label{endplem}
There exists a continuous function $f\colon \Lelfan \to \Lelfan$, which is not one-to-one, such that
 $f\rest E(\Lelfan)\colon E(\Lelfan)\to E(\Lelfan)$ is a homeomorphism. 
\end{lm}

\begin{pf}
Let $g\colon [0,1]^2\to [0,1]^2$ be a continuous function such that:
\begin{enumerate}[itemsep=0pt]
\item[(1)] for any $(x,y)\in [0,1]^2$, $\pi_1(g(x,y))=x$,
where $\pi_1$ is the projection onto the first coordinate;
\item[(2)] for any $x\in [0,1]$, $g(x,0)=(x,0)$;
\item[(3)] for every $x\neq \frac{1}{2}$, $g\rest (\{x\}\times [0,1])$ is one-to-one,
while $g\rest (\{\frac{1}{2}\}\times [0,1])$ is not one-to-one;
\item[(4)] $g(\frac{1}{2}, 1)\neq (\frac{1}{2},0)$.
\end{enumerate}
Such a function $g$ exists. We can, for example, take
\[
g(x,y) =
\begin{cases}
(x,2\alpha(x)y) & \qquad \text{if } y\leq \frac{1}{2}\\
(x, y+\alpha(x)-\frac{1}{2}) & \qquad \text{if } y\geq\frac{1}{2},
\end{cases}
\] where
$\alpha(x)=-x+\frac{1}{2}$ if $0\leq x\leq \frac{1}{2}$, and $\alpha(x)=x - \frac{1}{2}$ if $\frac{1}{2}\leq x\leq 1$.

Let $\pi\colon C\times [0,1]\to V$ be the quotient map.  (so $\pi(C\times\{0\})=(\frac{1}{2},0)$).
Let $f_1\colon \Lelfan \to V$ be given by $f_1(\pi(x,y))=\pi(g(x,y))$. By (1) and  (2) in the definition of $g$, $f_1$
is well defined. Since $g$ is continuous and not one-to-one, so is $f_1$. 

We claim that the image $f_1(\Lelfan)$ is homeomorphic to the Lelek fan.
By (1) and (2) in the definition of $g$, $f_1(\Lelfan)$ is a subfan of $V$. The set of end-points of $f_1(\Lelfan)$ is dense, because: 
the set of end-points of $\Lelfan$ is dense, $f_1$ is continuous, and  end-points of $\Lelfan$ are mapped to end-points of $f_1(\Lelfan)$.

We now verify that $f_1\rest (E(\Lelfan))\colon E(\Lelfan)\to E(f_1(\Lelfan))$ is a homeomorphism. By (3) and (4) in the definition of $g$, $f_1$ 
maps end-points to end-points and non-end-points to non-end-points, therefore $f_1\rest (E(\Lelfan))$ is a bijection
between  $E(\Lelfan)$ and $ E(f_1(\Lelfan))$.
Since $f_1$ is continuous, so is $f_1\rest (E(\Lelfan))$.
Using the compactness of $\Lelfan$, continuity of $f_1$,
and the fact that $f_1\rest (E(\Lelfan))$ is a bijection, we also see that $f_1\rest (E(\Lelfan))$  is open.

Let $f_2\colon f_1(\Lelfan) \to \Lelfan$ be a homeomorphism (it exists because of the uniqueness of the Lelek fan). Take $f=f_2\circ f_1$.
This $f$ is as required.
\end{pf}

\section{The Poulsen simplex}

In this section we present category-theoretic framework for metrizable Choquet simplices, showing that the Poulsen simplex is a \fra\ limit. It is worth pointing out that the geometric theory of finite-dimensional simplices is very similar to the theory of finite fans.

\subsection{Simplices}

A point $x$ in a compact convex set $K$ is an \emph{extreme point} if whenever $x=\lambda y +(1-\lambda) z$ for some $\lambda\in [0,1]$, $y,z\in K$, then $\lambda=0$ or $\lambda=1$. The set of extreme points of $K$ is denoted by $\ext K$.
For compact and convex sets $K$ and $L$,
recall that a map $p\colon L\to  K$ is \emph{affine} if for any $x,y\in L$ and $\lambda\in [0,1]$ we have
$p(\lambda x+(1-\lambda)y)=\lambda p(x) + (1-\lambda)p(y)$.
We call a map $p\colon L\to  K$ a \emph{projection} if it is an affine continuous surjection.
Note that in this case $\ext K \subs \img{p}{\ext L}$ and if $K$ is finite-dimensional then $p$ has a right inverse which is an affine embedding.
Furthermore, every continuous affine mapping defined on a compact convex set $L$ is uniquely determined by its restriction to $\ext L$.

A \emph{Choquet simplex} (later just called a  \emph{simplex}) is for us a non-empty compact convex and metrizable  set $K$ in a locally convex linear topological space such that every $x\in K$ has a unique representing measure, that is, a unique probability measure $\mu$ supported on the set of extreme points of $K$ and such that 
$$f(x) = \int_K f\, d\mu$$ for every continuous affine function $\map f K \Err$.
For more information on the theory of Choquet simplices we refer to Phelps' book~\cite{Phelps}.

Every finite-dimensional  simplex is, up to an affine isomorphism, of the form
$$\Delta_n = \setof{x \in \Err^{n+1}}{\sum_{i=1}^{n+1}x(i) = 1 \text{ and } x(i)\goe0 \text{ for every }i=1, \dots, n+1},$$
and $n\goe0$ is the \emph{dimension} of the simplex.
In particular, $\Delta_0$ is a singleton, $\Delta_1$ is a closed interval, and $\Delta_2$ is a triangle.

Let  $\map f {\Delta_{n+1}}{\Delta_n}$ be a projection.
Identifying $\Delta_n$ with
one of the $n$-dimensional faces of $\Delta_{n+1}$,
we can think of $f$ as a projection determined by choosing a point of $\Delta_n$ to be the image of the unique extreme point (vertex) of $\Delta_{n+1}$ that is not in $\Delta_n$.
Every projection $\map f{\Delta_n}{\Delta_m}$ is a compositions of such projections.
As a consequence, we see that every such map comes as the restriction of a linear projection $P_f$ from $\Err^{n+1}$ onto $\Err^{m+1}$.
Furthermore, when each $\Err^k$ is endowed with the Euclidean norm, all projections $P_f$ are 1-Lipschitz.
Thus, given an inverse sequence of projections
$$\xymatrix{
\Delta_{k_0} & \Delta_{k_1} \ar[l] & \Delta_{k_2} \ar[l] & \cdots \ar[l]
}$$
we can view its limit as a compact convex subset of the Hilbert space (the limit of the sequence of 1-Lipschitz projections between Euclidean spaces).

An \emph{embedding} of simplices is a one-to-one continuous affine map. By compactness,  every embedding is a topological embedding.
An embedding $\map i K L$ will be called \emph{stable} if $\img i {\ext K} \subs \ext L$.
We have the following analog of Lemma \ref{LmRaTxGh}.
\begin{lm}\label{stable}
Let $L$ be a finite simplex and let $\map p L K$ be a projection of simplices.
Then there exists a stable embedding $\map i K L$ such that $p \cmp i = \id K$.
\end{lm}

%(one direction)

\subsection{The construction and properties of the Poulsen simplex}
Let $\Sim$ be the category whose objects are finite-dimensional simplices
$$\Delta_0, \Delta_1, \Delta_2, \dots.$$ 
Maps between the simplices are projections and an arrow from $F$ to $G$ is a pair $\pair e p$ such that $\map e F G$ is a stable embedding,  $\map p G F$ is a projection and $p \cmp e = \id F$.
Note that every $p$ is $1$-Lipschitz with respect to the metrics inherited from the Euclidean metrics.

The following is well-known. %see also bottom of page 99 in LOS
\begin{tw}[Corollary to Thm. 5.2, \cite{LaLi}]
Metrizable Choquet simplices are, up to affine homeomorphisms, precisely the limits of inverse sequences in $\Sim$.
\end{tw}

Let us note the following important property of our category.

\begin{lm}\label{LmRevAmlgsms}
The category $\Sim$ has the strict amalgamation property.
\end{lm}

\begin{pf}
We may assume that $f$ and $g$ are projections
of $\Delta_\ell$ and $\Delta_m$, respectively, onto $\Delta_k$, so that $\Delta_k$ is a face of both $\Delta_\ell$ and $\Delta_m$.
As $\Delta_\ell$ and $\Delta_m$ are abstract simplices, we may assume that they both ``live'' in the same vector space so that $\Delta_k$ is their intersection.
Now the convex hull of $\Delta_\ell \cup \Delta_m$ is a simplex, which we denote by $\Delta_n$, where $n = \ell+m - k$. Note that $f$, $g$ determine projections $f'$, $g'$ satisfying $f' \rest \Delta_m = g$, $f' \rest \Delta_\ell = \id{\Delta_\ell}$, and $g' \rest \Delta_\ell = f$, $g' \rest \Delta_m = \id{\Delta_m}$.
These projections obviously satisfy $f \cmp f' = g \cmp g'$.
\end{pf}

\begin{lm}
$\Sim$ is a separable metric category. % (in the sense of \cite{KubMF}).
\end{lm}

\begin{pf}
Note that $\ob{\Sim} = \sett{\Delta_n}{\ntr}$ is countable.
Denote by $\Sim^\Qyu$ the subcategory of $\Sim$ such that $\ob {\Sim^\Qyu} = \ob{\Sim}$ and an $\Sim^\Qyu$-arrow from $\Delta_m$ to $\Delta_n$ is an $\Sim$-arrow 
$f=\pair e p $ such that $p(v)$ has rational barycentric coordinates for every extreme point $v$ in $\Delta_n$.
Then  $\Sim^\Qyu$ has countably many arrows and is as required.
\end{pf}

By Theorem \ref{fraseq}, $\Sim$ has a \fra\ sequence, whose limit we denote by $ U_\infty$.
The proof of the following theorem goes along the lines of the proof of Theorem \ref{dense}.

\begin{tw}\label{krbjkebfjkebfjkb}
Let $\vec U$ be a sequence in $\Sim$ and let $K$ be its inverse limit.
The following properties are equivalent:
\begin{enumerate}
	\item[(a)] The set $\ext K$ is dense in $K$.
	\item[(b)] $\vec U$ is a \fra\ sequence.
\end{enumerate}
\end{tw}

\begin{wn}
$U_\infty$ is the Poulsen simplex.
\end{wn}

From now on, we shall write $\mathbb{P}$ instead of $U_\infty$.
\fra\ theory  provides  direct proofs of the uniqueness, universality, and almost homogeneity, originally proved 
in 1978 by Lindenstrauss, Olsen and Sternfeld \cite{LOS}.

A subset $F$ of a simplex $K$ is called a \emph{face} if it is compact, convex, and $\ext F\subset \ext K$.

\begin{wn}[Uniqueness]	
The Poulsen simplex is unique, up to affine homeomorphisms.
\end{wn}

\begin{wn}[Universality]\label{CorSDFohesim}
Every metrizable simplex is affinely homeomorphic to a face of $\mathbb{P}$.
\end{wn}

\begin{wn}[Almost homogeneity]\label{almhomsim}
Let $F$ be a finite-dimensional face of a Poulsen simplex and let $f, g\colon \mathbb{P}\to F$ be projections.
Then for every $\eps>0$  there is an affine homeomorphism $H\colon\mathbb{P}\to\mathbb{P}$ such that for every 
$x\in \mathbb{P}$,  $d_F(f\circ H(x), g(x))<\eps$, where $d_F$ is any compatible metric on $F$.
\end{wn}

Similarly as for the Lelek fan (Corollary \ref{almhoms}) from Corollary \ref{almhomsim} we can deduce the following stronger statement.
\begin{wn}[Almost homogeneity]\label{almhomsims}
Let $F$ be a Choquet simplex simplex and let $f, g\colon \mathbb{P}\to F$ be projections.
Then for every $\eps>0$  there is an affine homeomorphism $H\colon\mathbb{P}\to\mathbb{P}$ such that for every 
$x\in \mathbb{P}$,  $d_F(f\circ H(x), g(x))<\eps$, where $d_F$ is a fixed compatible metric on $F$.
\end{wn}

The authors of \cite{LOS} proved a stronger homogeneity property: For every two closed faces $F$ and $G$ 
of the Poulsen simplex and an affine homeomorphism $h\colon F\to G$
there is an affine homeomorphism $H\colon\mathbb{P}\to\mathbb{P}$ extending $h$.

\section{Final remarks}

We say that a space $K$ is \emph{retractively universal} for a given class $\Pee$ of spaces, if for every $P \in \Pee$ there exists a retraction of $K$ whose image is homeomorphic to $P$.

A general question is which compact second countable spaces can be represented as inverse limits of sequences of retractions onto ``simpler" spaces.
A typical meaning of ``simple" could be a ``polyhedron".

\begin{prop}\label{univer}
	No metrizable compact space is retractively universal for the class of all non-empty 1-dimensional metrizable compact spaces.
\end{prop}

\begin{pf}
The details are presented in \cite[Cor. 3.2.4]{Calka}, where it  is shown that no  
second countable compact space is retractively universal for the class of subcontinua of a Cook continuum.
A {\em Cook continuum}, constructed by Cook~\cite{Cook},  is a hereditary indecomposable continuum  $C$ of dimension one, such that
there exists an uncountable family $\mathcal{F}$ of subcontinua of $C$ with the property that the only continuous map from a continuum in $\mathcal{F}$
to another continuum in $\mathcal{F}$ is a constant map.

Suppose now that $K$ is a metric compact space containing each element of $\Ef$ as a retract, so we may assume that $\Ef$ is a family of compact retracts of $K$.
Looking at it as a separable metric space with the Hausdorff distance, we find $\eps>0$ and a non-trivial sequence $\ciag F \subs \Ef$ converging to some $F \in \Ef$ and such that $\diam(F_n) \goe \eps$ for $\ntr$. Then also $\diam(F) \goe \eps$. Let $\map r K F$ be a retraction. Then $r \rest F_n$ is constant for each $\ntr$, showing that $r$ cannot be continuous.
\end{pf}

\end{document}